\documentclass{article}     
\usepackage{amsmath,amssymb,amsthm,enumerate,graphicx,centernot,bm,comment,stmaryrd}
\usepackage[all]{xy}

\newcommand{\ov}[1]{\overline{#1}}
\newcommand{\F}{\mathbb{F}}
\newcommand{\N}{\mathbb{N}}
\newcommand{\Z}{\mathbb{Z}}
\newcommand{\T}{\mathbb{T}}

\newcommand{\Is}{\mathcal{I}}
\newcommand{\Js}{\mathcal{J}}

\newcommand{\Mr}{\mathfrak{M}_r}
\newcommand{\Ml}{\mathfrak{M}_l}

\newcommand{\rank}{\operatorname{rank}}
\newcommand\numberthis{\addtocounter{equation}{1}\tag{\theequation}}

\newtheorem{theorem}{Theorem}[section]
\newtheorem{lemma}[theorem]{Lemma}
\newtheorem{corollary}[theorem]{Corollary}
\newtheorem{proposition}[theorem]{Proposition}
\newtheorem{definition}{Definition}[section]
\theoremstyle{definition}
\newtheorem*{remark}{Remark}
\newtheorem*{example}{Example}

\begin{document}

\title{Matroidal Root Structure of Skew Polynomials over Finite Fields\thanks{The authors were partially supported by the National Science Foundation under grant DMS-1547399.}
}

\author{Travis Baumbaugh \qquad Felice Manganiello\\
		tbaumba@clemson.edu \qquad manganm@clemson.edu\\
		Department of Mathematical Sciences,\\
		Clemson University\\
		Clemson, SC 29634-0975, USA
}

\maketitle

\begin{abstract}
A skew polynomial ring $R=K[x;\sigma,\delta]$ is a ring of polynomials with non-commutative multiplication. This creates a difference between left and right divisibility, and thus a concept of left and right evaluations and roots. A polynomial in such a ring may have more roots than its degree, which leads to the concepts of closures and independent sets of roots. There is also a structure of conjugacy classes on the roots. In $R=\F_{q^m}[x,\sigma]$, this leads to the matroids $\Mr$ and $\Ml$ of right independent and left independent sets, respectively. The matroids $\Mr$ and $\Ml$ are isomorphic via the extension of the map $\phi:[1]\to[1]$ defined by $\phi(a)=a^{\llbracket m\rrbracket}$, where $\llbracket i\rrbracket=\frac{q^{i-1}-1}{q-1}$ is a notation introduced to simplify the exponents in evaluation polynomials. Additionally, extending the field of coefficients of $R$ results in a new skew polynomial ring $S$ of which $R$ is a subring, and if the extension is taken to include roots of an evaluation polynomial of $f(x)$ (which does not depend on which side roots are being considered on), then all roots of $f(x)$ in $S$ are in the same conjugacy class.

\noindent Keywords {Matroids, Skew Polynomial Rings, Finite Fields, Isomorphism}

\noindent{12E20 \and 16S36 \and 05B35}
\end{abstract}

\section{Introduction}
\label{introduction}
When operating in a skew polynomial ring $R$, there are several key differences from a commutative polynomial ring. Factorizations are not unique $\cite{OrePaper}$, and a polynomial may have more roots than its degree. Because of non-commutativity, there is a difference between left and right divisibility, and evaluation is not as straight-forward as in the commutative case $\cite{LamLeroy}$. These properties make skew polynomial rings useful for coding \cite{BoucherCoding}, secret sharing \cite{SecretSharing}, and even cryptographic key exchange \cite{BoucherCrypto}. The relation between skew polynomial rings and skew fields is explored in \cite{CohnFields}, along with an introduction to basic properties. This paper delves into some of the structure of polynomials' left and right roots in a skew polynomial ring when viewed as elements of a larger skew polynomial ring.

In Chapter 2, some background on skew polynomials and their properties are given. The process for evaluating a skew polynomial at a value on either side is described. The evaluation polynomial $f(y)=\sum_{i=1}^nf_iy^{\frac{q^i-1}{q-1}}$ is defined. Some of the structure of the roots is also described in terms of the independence of a set of roots and the closure of a set of roots. A conjugacy relationship is defined so that we may speak of the conjugacy class $[a]$ of all elements conjugate to $a\in\F_{q^m}$, and the structure of these conjugacy classes as group cosets is discussed.

In Chapter 3, we first examine the matroidal structure of the sets of independent elements. The set of all right independent sets $\Is$ defines a matroid $\Mr=(\F_{q^m},\Is)$, and likewise the set of all left independent sets $\Js$ defines a matroid $\Ml=(\F_{q^m},\Js)$. We show that these two matroids are actually isomorphic via a map that, for any $a$ in the conjugacy class $[1]$ is defined as $\phi(a)=a^{\frac{q^{m-1}-1}{q-1}}$. As seen in \cite{Matroid}, these matroids may be used in the network coding structure of \cite{Network}.

Finally, in Chapter 4, we examine polynomials in the skew polynomial ring $R=\F_{q^m}[x;\sigma]$ as elements of a larger ring $S=\T[x;\gamma]$, where $\T$ is a finite extension of $\F_{q^m}$ and $\gamma$ restricted to $\F_{q^m}$ is $\sigma$. We discuss what this means for the structure of the roots. Finally, a specific extension field is used to construct the larger ring such that the polynomial has the maximum number of right roots, all of which are in a single class. This splitting field is related to but may be distinct from the splitting field of the linearized polynomial. We also show that it results in the maximum number of left roots, while nothing is gained from a larger extension. This is followed by a review of the results so far.

\section{Background on Skew Polynomial Rings}
\label{background}
Before examining skew polynomials in extension fields, it is necessary to set down some fundamental definitions and properties.

Skew polynomials rings were introduced by Oyestein Ore in the 1933 paper, ``Theory of Non-Commutative Polynomials'' $\cite{OrePaper}$.

\begin{definition}
	Given a division ring $K$, an injective homomorphism \linebreak ${\sigma:K\to K}$, and a $\sigma$-derivation $\delta$ (an additive homomorphism $\delta:K\to K$ such that $\delta(ab)=\sigma(a)\delta(b)+\delta(a)b$ for all $a,b\in K$), let $R$ be the set of polynomials of the form $f(x)=\sum_{i=0}^na_ix^i$, where $n\in\N$ and $a_i\in K$ for $i\in\{0,\dots,n\}$. The skew polynomial ring denoted $R=K[x;\sigma,\delta]$ is this set of polynomials with standard addition, but with multiplication determined by the rule
	\[xa=\sigma(a)x+\delta(a)\]
	for all $a\in K$. We will call the elements of $R$ skew polynomials.
\end{definition}

Later, we take $\delta\equiv0$, that is $\delta(a)=0$ for all $a\in K$, and so we write $R=K[x;\sigma]$. In fact, we also generally take $K=\F_{q^m}$, and so $\sigma$ is not only a homomorphism, but an automorphism. This is the case with the ring in the following example, which will be used throughout the paper.

\begin{example}
	Let $K=\F_{3^2}=\F_3[\alpha]/(\alpha^2+2\alpha+2)$, $\sigma(a)=a^3$, and $\delta\equiv0$. Then $\alpha$ is a primitive element of $K$. We let $R=K[x;\sigma]$ (since $\delta\equiv0$). Then for instance:
	\begin{align*}
	(x+1)(x+\alpha)&=x^2+\alpha^6x+\alpha,
	\end{align*}
	but reversing the order of the factors, we have
	\begin{align*}
	(x+\alpha)(x+1)&=x^2+\alpha^2x+\alpha.
	\end{align*}
\end{example}

As seen in the example, multiplication in a skew polynomial ring is not in general commutative (in fact, it is commutative if and only if $\sigma$ is the identity homomorphism and $\delta\equiv0$). However, since $K$ is a division ring, there are no zero divisors, and so if $a,b\in K$ with $a,b\ne0$, then because $\sigma^n(b)\ne0$, we have $a\sigma^n(b)\ne0$. This means that by applying the multiplication rule $m$ times, the product of $ax^n\cdot bx^m$ will have a leading term $a\sigma^{n}(b)x^{m+n}$, and so if two polynomials $f(x),g(x)\in R$ have degrees $n$ and $m$ respectively, their product will have degree $m+n$. By considering the implications of multiplication adding degrees on polynomials of degree $0$ and $1$ under typical assumptions of associativity and distribution of multiplication, one sees that the definition of skew polynomial rings is in fact the most general definition such that this is the case ($\cite{OrePaper}$, equation 3).

We now speak about the divisibility of polynomials in a skew polynomial ring $R$. This is first shown in $\cite{OrePaper}$:

\begin{lemma}
	A skew polynomial ring $R$ is a right Euclidean ring and, if $\sigma$ is surjective, a left Euclidean ring.
\end{lemma}

That is, given any two polynomials $f(x),g(x)\in R$, there exist unique \linebreak$q_r(x),r_r(x)\in R$ such that
\[f(x)=q_r(x)g(x)+r_r(x),\]
where $\deg r_r(x)<\deg g(x)$ or $r_r(x)=0$. An algorithm for this is also laid out in \cite{EuclideanAlg}, while the problem of fast factoring in general is addressed in \cite{Giesbrecht} and \cite{Factoring}. Likewise, if $\sigma$ is surjective, there exist unique $q_l(x),r_l(x)\in R$ such that
\[f(x)=g(x)q_l(x)+r_l(x),\]
where $\deg r_l(x)<\deg g(x)$ or $r_l(x)=0$.

\begin{example}
	Working in the ring $R$ as before, if $f(x)=x^2+\alpha^5x+\alpha^7$ and $g=(x-1)$, then
	\[f(x)=(x+\alpha^3)g(x)+0,\]
	so $q_r(x)=x+\alpha^3$ and $r_r(x)=0$, and
	\[f(x)=g(x)(x+\alpha)+\alpha^6,\]
	so $q_l(x)=x+\alpha$ and $r_l(x)=\alpha^6$.
\end{example}

\begin{definition}
	If $r_r(x)=0$, we say that $g(x)$ divides $f(x)$ on the right and write $g(x)|_rf(x)$. If $r_l(x)=0$, we say that $g(x)$ divides $f(x)$ on the left and write $g(x)|_lf(x)$.
\end{definition}

In the previous example, for instance, we had $g(x)\!\mid_r\!f(x)$, but
$g(x)\!\nmid_l\!f(x)$. From these definitions of divisibility, we define the concept of evaluation in a skew polynomial ring.

When evaluating a polynomial $f(x)\in R$, we desire that the value of the polynomial evaluated at $a\in K$ is the remainder of division by the polynomial $(x-a)$. Since $R$ is non-commutative, however, left and right evaluations are different in general.

\begin{definition}
	Given any polynomial $f(x)\in R$ and any $a\in K$, we may write $f(x)=q_r(x)(x-a)+r$. Then the evaluation of $f(x)$ on the right at $a$ is $f(a)_r=r$. Likewise, if $\sigma$ is surjective, then left division is possible and we may write $f(x)=(x-a)q_l(x)+s$, and the evaluation of $f(x)$ on the left at $a$ is $f(a)_l=s$.
\end{definition}

\begin{example}
	For instance, returning to a previous example, if we have \linebreak$f(x)=x^2+\alpha^5x+\alpha^7$, then $f(1)_r=0$ and $f(1)_l=\alpha^6$.
\end{example}

Thus we find that $f(a)_r=0$ if and only if $(x-a)|_rf(x)$, and similarly $f(a)_l=0$ if and only if $(x-a)|_lf(x)$.

To have this property for evaluation, it is not possible to simply ``plug in'' $a$ in place of $x$ in the polynomial and apply the powers, multiplication by coefficients, and addition. Rather, a new formula for evaluation outlined in $\cite{LamLeroy}$ is necessary:

\begin{theorem}
	Define recursively
	\begin{align*}
	N_0(a)&=1\\
	N_{i+1}(a)&=\sigma(N_i(a))a+\delta(N_i(a))\numberthis\label{rec1}
	\end{align*}
	for all $a\in K$ and $i\ge1$. For $a\in K$ and any polynomial $f(x)=\sum_{i=0}^d f_ix^i\in R$, we have $f(a)_r=\sum_{i=0}^df_iN_i(a)$.
\end{theorem}

Similarly, if $\sigma$ is an automorphism, the following holds.

\begin{theorem}
	Define recursively
	\begin{align*}
	M_0(a)&=1\\
	M_{i+1}(a)&=a\sigma^{-1}(M_i(a))-\delta(\sigma^{-1}(M_i(a)))\numberthis\label{rec2}
	\end{align*}
	for all $a\in K$ and $i\ge1$. For $a\in K$ and any polynomial $f(x)=\sum_{i=0}^n f_ix^i\in R$, we rewrite $f(x)=\sum_{i=0}^n x^if_i'\in R$ and have $f(a)_l=\sum_{i=0}^nM_i(a)f_i'$.
\end{theorem}

We also have the following result which relates left evaluation of polynomials in the ring $R$ to right evaluation of a different polynomial in the ring $R'$. 

\begin{corollary}\label{leftasright}
	Let $F$ be a field. If we define $\sigma'=\sigma^{-1}$ and $\delta'=-\delta\circ\sigma^{-1}$, then $R'=F[x;\sigma',\delta']$ is a skew polynomial ring, and for any $f(x)=\sum_{i=0}^df_ix^i$ rewritten as $f(x)=\sum_{i=0}^d x^if_i'\in R$ and any $a\in F$, we have $f(a)_l=f'(a)_r$, where $f'(x)\in R'$ is the polynomial $f'(x)=\sum_{i=0}^df_i'x^i$.
\end{corollary}

We see that this is because in $R'=F[x;\sigma',\delta']$, for $i\ge1$, we have \[N_{i+1}(a)=\sigma'(N_i(a))a+\delta'(N_i(a))=\sigma^{-1}(N_i(a))a-\delta(\sigma^{-1}(N_i(a))),\]
while in $R=F[x;\sigma,\delta]$, we have $M_{i+1}(a)=a\sigma^{-1}(M_i(a))-\delta(\sigma^{-1}(M_i(a)))$, and since multiplication in fields is commutative, these recurrence relations are identical, resulting in the equivalence between right evaluation in $R'$ and left evaluation in $R$ (the modified polynomial $f'(x)$ takes care of the requirement that the coefficients be on the left). That $R'$ is actually a skew polynomial ring is verified by checking that $\sigma'$ and $\delta'$ satisfy the required properties.

In fact, since for any polynomial $f'(x)\in R'$, we may find a polynomial \linebreak$f(x)\in R$ such that $f'(a)_r=f(a)_l$, we have that there is an equivalence between left evaluation in $R$ and right evaluation in $R'$.

\begin{remark}
In our work, we take $K=\F_{q^m}$, $\sigma$ to be the Frobenius automorphism, and  $\delta\equiv0$. That is, $\sigma(a)=a^q$ for all $a\in K$, so $\sigma^{-1}(a)=a^{q^{m-1}}$. Thus $\eqref{rec1}$ simplifies to
\begin{align*}
N_0(a)&=1
\\N_{i+1}(a)&=(N_i(a))^qa,
\end{align*}
the solution to which is $N_i(a)=a^{\frac{q^i-1}{q-1}}$. To write this more compactly, we introduce the notation $\llbracket i\rrbracket=\frac{q^i-1}{q-1}$, and so we may write $N_i(a)=a^{\llbracket i\rrbracket}$.

Likewise, the recursive formula for $M$ simplifies to
\begin{align*}
M_0(a)&=1
\\M_{i+1}(a)&=a(M_i(a))^{q^{m-1}},
\end{align*}
with the solution $M_i(a)=a^{\frac{q^{i(m-1)}-1}{q^{m-1}-1}}$. For brevity we shall also introduce the notation $\rrbracket i\llbracket=\frac{q^{i(m-1)}-1}{q^{m-1}-1}$ so we may write this as $a^{\rrbracket i\llbracket}$.
\end{remark}

If $f(x)=\sum_{i=0}^nf_ix^i$, then we may write
\[f(a)_r=\sum_{i=0}^nf_ia^{\llbracket i\rrbracket}.\]
Thus, if we let $\ov{f^r}(y)=\sum_{i=0}^nf_iy^{\llbracket i\rrbracket}\in\F_{q^m}[y]$, we have $f(a)_r=\ov{f^r}(a)$. In fact, this allows us to define a map:
\begin{align*}
\ov{\hspace{0.1in}^r}:\F_{q^m}[x;\sigma]&\to\F_{q^m}[y]\\
x^i&\mapsto y^{\llbracket i\rrbracket}
\end{align*}

This map is an injective linear map with respect to polynomial addition and scalar multiplication. For any $f(x)\in R$ it holds that for any $a\in\F_{q^m}$, \linebreak$f(a)_r=\ov{f^r}(a)$. This leads to the following definition.
\begin{definition}
	For any $f(x)\in\F_{q^m}[x;\sigma]$, the polynomial $\ov{f^r}(y)\in\F_{q^m}[y]$ is called the right evaluation polynomial of $f(x)$.
\end{definition}

Similarly, we may write
\[f(a)_l=\sum_{i=0}f_i'a^{\rrbracket i\llbracket},\]
and so we define a map:
\begin{align*}
\ov{\hspace{0.1in}^l}:\F_{q^m}[x;\sigma]&\to\F_{q^m}[y]\\
ax^i&\mapsto\sigma^{-i}(a)y^{\rrbracket i\llbracket}
\end{align*}
with the property that for any $f(x)\in R$ and any $a\in\F_{q^m}$, $f(a)_l=\ov{f^l}(a)$.

\begin{remark}
This is not a linear map due to the $\sigma^{-i}(a)$. However, it results in a similar definition as with right evaluation.
\end{remark}

\begin{definition}
	For any $f(x)\in\F_{q^m}[x;\sigma]$, the polynomial $\ov{f^l}(y)\in\F_{q^m}[y]$ is called the left evaluation polynomial of $f(x)$.
\end{definition}

At this point, we introduce a concept from $\cite{LamLeroy}$ necessary to evaluate the product of two polynomials. For any $a\in K$ and $c\in K^*$, we let
\[a^c=(\sigma(c)a+\delta(c))c^{-1},\]
which in the case of $K=\F_{q^m}$, $\sigma(a)=a^q$, and $\delta\equiv0$ simplifies to $a^c=ac^{q-1}$. This defines an equivalence relation on $\F_{q^m}$, where two elements $a,b\in K$ are $\sigma$-conjugates if there is some $c\in K^*$ such that $a^c=b$. As seen in $\cite{Matroid}$, $0$ is conjugate only with itself, and the size of the remaining conjugacy classes is $\llbracket m\rrbracket$, with $q-1$ distinct conjugacy classes.

For notation, we let $[a]$ denote the conjugacy class of $a\in\F_{q^m}$. That is, it is the set of all elements of $\F_{q^m}$ which are conjugate to $a$. We thus have $[0]=\{0\}$. Since $a$ and $b$ are $\sigma$-conjugates if there is some $c\in\F_{q^m}^*$ such that $ac^{q-1}=b$, we find that the class of $1$ is \[[1]=\{\alpha^0,\alpha^{q-1},\alpha^{2(q-1)},\dots,\alpha^{(\llbracket m\rrbracket-1)(q-1)}\},\]
which is the subgroup of $\F_{q^m}$ composed of the $q-1$ powers of elements of $\F_{q^m}^*$. Each of the remaining conjugacy classes is just the same elements multiplied by $\alpha^i$ for $i$ from $1$ to $q-2$. That is, \[[\alpha^i]=\{\alpha^i,\alpha^{q-1+i},\alpha^{2(q-1)+i},\dots,\alpha^{(\llbracket m\rrbracket-1)(q-1)+i}\},\]
and so the remaining $q-1$ classes are the cosets of the class of $1$.

Since right conjugacy is defined based on $\sigma$ and $\delta$, with $\delta\equiv0$, we have that $a,b\in\F_{q^m}$ are conjugate if there is some $c\in\F_{q^m}^*$ such that $a^c=b$. That is, $ac^{q-1}=b$. For left conjugacy, we instead use $\sigma^{-1}$, and so $a,b\in\F_{q^m}$ are conjugate if there is some $c\in\F_{q^m}$ such that $ac^{q^{m-1}-1}=b$.

If $a$ and $b$ are right conjugate, then there is some $c\in\F_{q^m}^*$ such that $ac^{q-1}=b$, and so if we let $d=c^{-q}$, then
\begin{align*}
ad^{q^{m-1}-1}&
=a(c^{-q})^{q^{m-1}-1}
=a(c^{-1})^{q^m}c^q
=ac^{q-1}
=b,
\end{align*}
and so $a$ and $b$ are left conjugate. Similarly, if $a$ and $b$ are left conjugate, we have some $d\in\F_{q^m}^*$ such that $ad^{q^{m-1}-1}=b$, and so if we take $c=d^{\frac{q^{m-1}-1}{q-1}}$, then we clearly have
\begin{align*}
ac^{q-1}&=a\left(d^{\frac{q^{m-1}-1}{q-1}}\right)^{q-1}
=ad^{q^{m-1}-1}
=b,
\end{align*}
and so $a$ and $b$ are right conjugate. This proves that two elements are right conjugate if and only if they are left conjugate, and so the left and right conjugacy classes are identical.

This concept of conjugacy allows us to formulate the following theorem, proved in $\cite{LamLeroy}$, which allows us to evaluate products of polynomials.

\begin{theorem}
	If $h(x)=f(x)g(x)$ where $f(x),g(x)\in R$, then for any $a\in K$, if $g(a)_r=0$, $h(a)_r=0$, but if $g(a)_r\ne0$, we have $h(a)_r=f(a^{g(a)_r})_rg(a)_r$.	
\end{theorem}

We now introduce the idea of left and right minimal polynomials for a set.
\begin{definition}
	The right minimal polynomial of a set $Z=\{a_1,a_2,\dots,a_n\}$, with $a_i\in K$ for $i\in\{1,\dots,n\}$ is the monic polynomial $\mu_{Z,r}(x)\in R$ of minimal degree such that $\mu_{Z,r}(a_i)_r=0$ for all $i\in\{1,\dots,n\}$. Similarly, the left minimal polynomial of $Z$ is the monic polynomial $\mu_{Z,l}(x)\in R$ of minimal degree such that $\mu_{Z,l}(a_i)_l=0$ for all $i\in\{1,\dots,n\}$.
\end{definition}

Using the formula for evaluation of products, we interpolate polynomials with a given set of roots. Let $a_1,a_2,\dots,a_n$ be elements of $K$. Let $f_1(x)=x-a_1$. Then clearly $f_1$ has $a_1$ as a root, since $f_1(a_1)_r=0$. Then for $2\le i\le n$, we calculate $c_i=f_{i-1}(a_i)_r$. If $c_i=0$, then we take $f_i(x)=f_{i-1}(x)$, but if $c_i\ne0$, we take \[f_i(x)=(x-a_i^{c_i})f_{i-1}(x)=(x-a_ic_i^{q-1})f_{i-1}(x).\numberthis\label{interpolate}\]
Since $(x-a_i^{c_i})$ evaluated at $a_i^{c_i}$ is $0$, we have that $f_i(a_i)_r=0$. By this construction, we will have that $f_i(x)$ has $a_1,a_2,\dots,a_i$ as roots. This can be continued up to $i=n$ to construct the polynomial $f_n(x)$ with all of $a_1,\dots,a_n$ as roots. The proof of this theorem can be found in $\cite{Matroid}$.

\begin{theorem}\label{interp}
The polynomial $f_n(x)$ constructed above is the right minimal polynomial $\mu_{Z,r}(x)$ for $Z=\{a_1,a_2,\dots,a_n\}$.
\end{theorem}

\section{Structure of Roots}
\label{roots}
Now that the minimal polynomial has been defined, we define the closure of a set of elements.

\begin{definition}\label{d:independent}
	The right closure of $Z=\{a_1,a_2,\dots,a_n\}$, denoted $\ov{Z}^r$, is the set of all right roots of the right minimal polynomial of $Z$. That is,
	
	\[\ov{Z}^r=\{a\in\F_{q^m}\mid\mu_{Z,r}(a)_r=0\}.\]
	Likewise the left closure, $\ov{Z}^l$, of $Z$ is the set of all left roots of the left minimal polynomial of $Z$, meaning
	\[\ov{Z}^l=\{a\in\F_{q^m}\mid\mu_{Z,l}(a)_l=0\}.\]
\end{definition}

With these definitions in place, we define the independence of a set.

\begin{definition}\label{closure}
	A set $Z=\{a_1,a_2,\dots,a_n\}$ is called right independent if for any $i\in\{1,\dots,n\}$, $a_i\notin\ov{Z\setminus\{a_i\}}^r$. Similarly, $Z$ is left independent if $a_i\notin\ov{Z\setminus\{a_i\}}^l$ for all $i\in\{1,\dots,n\}$.
\end{definition}

We can also talk about the rank of a set.

\begin{definition}
	The right rank of a set $Z$ is the maximum cardinality of a right independent subset of $Z$, which is $\deg(\mu_{Z,r})$. Likewise, the left rank of the set $Z$ is the maximum cardinality of a left independent subset of $Z$, which is $\deg(\mu_{Z,l})$.
\end{definition}

We note that in general the left rank is not equal to the right rank.

\subsection{Matroidal Structure}
\label{matroid}
We define the sets that make up matroids for a skew polynomial ring.

\begin{theorem}
	Let $R=\F_{q^m}[x;\sigma]$ be a skew polynomial ring. Then let $\Is$ be the set of all right independent subsets of $\F_{q^m}$ and $\Js$ be the set of all left independent subsets of $\F_{q^m}$. Then $\Mr=(\F_{q^m},\Is)$ and $\Ml=(\F_{q^m},\Js)$ are matroids.
\end{theorem}

That $\Mr$ is a a matroid is proven in $\cite{Matroid}$. Since there is an equivalence between left evaluation in $R$ and right evaluation in $R'$, the left matroid $\Ml$ of $R$ is just the right matroid $\Mr$ of $R'$, and so the same method of proof applies to it.

With this structure in place, the independent sets of the matroid are the independent sets defined earlier. Likewise, the rank function of a set of elements in the matroid is the rank of the set of elements in $R$. The closures of the matroids are closures of the sets of elements in $R$, and finally we have that the flats for the matroids are exactly those sets that are equal to their closures. That is, for $\Mr$, the flats are all subsets $Z\subset\F_{q^m}$ such that $\ov{Z}^r=Z$, and for $\Ml$, the flats are all subsets $Z\subset\F_{q^m}$ such that $\ov{Z}^l=Z$.

\subsection{Matroid Isomorphism}
\label{isomorphism}
In general, $\Is\ne\Js$, and so $\Mr$ and $\Ml$ are not equal. When $m=2$, then $\sigma^{-1}=\sigma$, and so the left and right closures are the same. Since dependency may be defined based on closures, right and left dependencies are the same and so $\Is=\Js$ when $m=2$. If $m\ne2$, then necessarily $\sigma^{-1}\ne\sigma$, so we can find closures that are different and thus a set that is right independent but not left independent, and so $\Is\ne\Js$. However, we show that $\Mr$ and $\Ml$ are isomorphic.

First, note that for any $a\in\F_{q^m}$, we may consider the class $[a]$ and the set $\Is|[a]$ of sets $X\in\Is$ such that $X\subseteq[a]$. We have that $[a]_r=([a],\Is|[a])$ is a sub-matroid of $\Mr$. Similarly, $[a]_l=([a],\Js|[a])$ is a sub-matroid of $\Ml$. Since the conjugacy classes are $[0]$ and $[\alpha^i]$ for $0\le i\le q-2$, there are $q$ distinct right sub-matroids: $[0]_r$ and $[\alpha^i]_r$ for $0\le i\le q-2$. There are also $q$ left sub-matroids: $[0]_l$ and $[\alpha^i]_l$ for $0\le i\le q-2$. We will construct the isomorphism from $\Mr$ to $\Ml$ from isomorphisms between these.

\begin{proposition}
For any $0\le i\le q-2$, Let $\gamma_i:[1]\to[\alpha^i]$  be defined by $\gamma_i(a)=\alpha^ia$. Then $\gamma_i$ is a matroid isomorphism from $[1]_r$ to $[\alpha^i]_r$ and $[1]_l$ to $[\alpha^i]_l$.
\end{proposition}

\begin{proof}
We note that if $a\in[1]$, then there is some $c\in\F_{q^m}^*$ such that $1c^{q-1}=a$, and so $\alpha^ic^{q-1}=\alpha^ia$, which means $\gamma_i(a)\in[\alpha^i]$, and $\gamma_i$ is indeed a function from $[1]$ to $[\alpha^i]$. Note that if $\gamma_i(a)=\gamma_i(b)$, then $\alpha^ia=\alpha^ib$, and so $a=b$, which means $\gamma_i$ is one-to-one. Since $[1]$ and $[\alpha^i]$ are finite sets of the same size, $\gamma_i$ is therefore a bijection.

Now consider any set $Z\subset[1]$. Let $\mu_{Z,r}(x)$ be the minimal polynomial of $Z$ on the right. Then if $\mu_{Z,r}(x)=\sum_{j=1}^df_jx^j,$ let $g(x)=\sum_{j=1}^df_j\alpha^{-i\llbracket j\rrbracket}x^j$. For any $a\in Z$, we have
\begin{align*}
g(\alpha^ia)_r&=\sum_{j=1}^df_j\alpha^{-i\llbracket j\rrbracket}(\alpha^ia)^{\llbracket j\rrbracket}
=\sum_{j=1}^df_ja^{\llbracket j\rrbracket}
=\mu_{Z,r}(a)_r
=0,
\end{align*}
and so $\gamma_i(Z)$ are roots of $g(x)$. Similar work shows that if there were some polynomial $g'(x)$ of lesser degree with all of $\gamma_i(Z)$ as roots, it could be transformed into a polynomial with all of $Z$ as right roots and a lower degree than $\mu_{Z,r}(x)$, so $g(x)$ is the minimal polynomial of $\gamma_i(Z)$. Since $Z$ is right independent if and only if $\deg(\mu_{Z,r})=|Z|$ and $\deg(\mu_{Z,r})=\deg(\mu_{\gamma_i(Z),r})=\deg(g)$, we have that $\deg(g)=|Z|=|\gamma_i(Z)|$ if and only if $Z$ is right independent, which means $\gamma_i(Z)$ is right independent if and only if $Z$ is right independent.

Using the same steps, but by transforming the polynomial with $\alpha^{i\rrbracket j\llbracket}$ instead, we also find that $\gamma_i(Z)$ is left independent if and only if $Z$ is left independent. Thus, $[1]_r$ is isomorphic to $[\alpha^i]_r$ via the isomorphism $\gamma_i$, and $[1]_l$ is isomorphic to $[\alpha^i]_l$ via $\gamma_i$.
\end{proof}

We prove some important incremental results that allow us to construct another isomorphism. The first is largely a rephrasing of Lemma 1 from $\cite{Matroid}$:

\begin{lemma}\label{rclose}
Let $Z=\{a_1,\dots,a_n\}\subseteq[1]$ and $b_1,\dots,b_n\in\F_{q^m}$ such that $a_i=b_i^{q-1}$ for $i\in\{1,\dots,n\}$. Then
\[\ov{Z}^r=\left\{\left.\left(\sum_{j=1}^nc_jb_j\right)^{q-1}\right| c_j\in\F_q\right\}\setminus\{0\}.\]
\end{lemma}

\begin{proof}
If we let $b=\sum_{j=1}^nc_jb_j$, where $c_j\in\F_q$ for $1\le j\le n$, then if $\mu_{Z,r}(x)=\sum_{i=1}^df_ix^i$ we write
\begin{align*}
b\mu_{Z,r}(b^{q-1})_r
&=b\sum_{i=1}^df_i(b^{q-1})^{\frac{q^i-1}{q-1}}
=\sum_{i=1}^df_i\left(\sum_{j=1}^nc_jb_j\right)^{q^i}\numberthis\label{linear1}\\
&=\sum_{j=1}^nc_jb_j\sum_{i=1}^df_i\left(b_j^{q-1}\right)^{\frac{q^i-1}{q-1}}
=\sum_{j=1}^nc_j\alpha^{k_j}\ov{\mu}_{Z,r}^r(a_j)
=0,
\end{align*}

where $k_j$ is such that $\alpha^{k_j}$ is the representation of $b_j$ in terms of the primitive element. The third equality holds since $c_j\in\F_q$, which means $c_j^{q^i}=c_j$.

As long as $b\ne 0$, this means that $b^{q-1}$ is a right root of $\mu_{Z,r}(x)$, and so $b^{q-1}=\left(\sum_{j=1}^nc_jb_j\right)^{q-1}$ is in the right closure of $Z$.

For the other direction, note first that by Proposition 2.6 in $\cite{Wedderburn}$, any element in $\ov{Z}^r$ must be conjugate to some $a_i\in Z$, and so must be in $[1]$. Suppose that $\deg(\mu_{Z,r}(x))=d$ as above. Then in constructing $\mu_{Z,r}$ by Formula $\eqref{interpolate}$, we will find a set of $d$ right independent elements $B_Z=\{a_{k_1},\dots,a_{k_d}\}\subseteq Z$, and by construction, $\mu_{B_Z,r}(x)=\mu_{Z,r}(x)$. Also by construction, we have that $a_{k_s}$ is not a root of the right minimal polynomial of $\{a_{k_1},\dots,a_{k_{s-1}}\}$. This in turn means that $b_{k_s}$ must not be an $\F_q$-linear combination of $\{b_{k_1},\dots,b_{k_{s-1}}\}$, or as established above, $a_{k_s}=b_{k_s}^{q-1}$ would have to be a root of the minimal polynomial, a contradiction. Thus, we have that $\{b_{k_1},\dots,b_{k_d}\}$ is $\F_q$-linearly independent and so spans a vector space of size $q^d$ that are solutions to $\sum_{i=1}^df_ix^{q^i}=0$ by $\eqref{linear1}$.

There are only $q^d$ solutions of $\sum_{i=1}^df_ix^{q^i}=0$, and so all solutions must be \linebreak$\F_q$-linear combinations of $\{b_{k_1},\dots,b_{k_d}\}$. Since for any $a\in[1]$ there is some $b\in\F_{q^m}^*$ such that $a=b^{q-1}$, if $a$ is a root of $\mu_{Z,r}(x)$, we could work backwards to $\eqref{linear1}$ to find that $b$ is a solution of $\sum_{i=1}^df_ix^{q^i}=0$, and so must be of the form
$b=\sum_{j=1}^nc_jb_j$, which finally means $a=\left(\sum_{j=1}^nc_jb_j\right)^{q-1}$, and we have the other containment.
\end{proof}

\begin{lemma}\label{lclose}
Let $Z=\{a_1,\dots,a_n\}\subseteq[1]$ and $b_1,\dots,b_n\in\F_{q^m}$ such that \linebreak$a_i=b_i^{q^{m-1}-1}$ for $i\in\{1,\dots,n\}$. Then
\[\ov{Z}^l=\left\{\left.\left(\sum_{j=1}^nc_jb_j\right)^{q^{m-1}-1}\right| c_j\in\F_q\right\}\setminus\{0\}.\]
\end{lemma}

\begin{proof}
If we let $b=\sum_{j=1}^nc_jb_j$, where $c_j\in\F_q$ for $1\le j\le n$, then if $\mu_{Z,l}(x)=\sum_{i=1}^dx^if_i'$ we write
\begin{align*}
b\mu_{Z,l}(b^{q^{m-1}-1})_r
&=b\sum_{i=1}^d(b^{q^{m-1}-1})^{\frac{(q^{m-1})^i-1}{q^{m-1}-1}}f_i'
=\sum_{i=1}^df_i'\left(\sum_{j=1}^nc_jb_j\right)^{q^{i(m-1)}}\numberthis\label{linear2}\\
&=\sum_{j=1}^nc_j\alpha^{k_j}\sum_{i=1}^df_i'\left(\alpha^{k_j(q^{m-1}-1)}\right)^{\frac{q^{i(m-1)}-1}{q^{m-1}-1}}\\
&=\sum_{j=1}^nc_j\alpha^{k_j}\ov{\mu_{Z,l}^l}(a_j)
=0,
\end{align*}
where as before $k_j$ is such that $b_j=\alpha^{k_j}$, and $c_j^{q^{i(m-1)}}=c_j$ because $c_j\in\F_q$. Just like in the right case, so long as $b\ne 0$ (so some $c_j\ne0$), we have that $b^{q^{m-1}-1}$ is a left root of $\mu_{Z,l}(x)$, which means that $b^{q^{m-1}-1}$ is in the left closure of $Z$.

For the other direction, the proof is essentially the same as in the right case. The key note is that $\sum_{i=1}^df_ix^{q^{m-1}-1}$ is an $(m-1)$-linearized polynomial, and so by Theorem 4 of $\cite{Matroid}$, there are also only $q^d$ solutions to $\eqref{linear2}$, so once we find $d$ $\F_q$-linearly independent $b_i$'s, we must have that all roots are a $(q^{m-1}-1)$th power of some $\F_q$-linear combination of them.
\end{proof}

\begin{theorem}
Let $\phi:[1]\to[1]$ be defined by $\phi(a)=a^{\llbracket m-1\rrbracket}$ for all $a\in[1]$. Then $\phi$ is a matroid isomorphism from $[1]_r$ to $[1]_l$.
\end{theorem}

\begin{proof}

We start by showing that $\phi$ is a bijection from $[1]$ to $[1]$. For any $a\in[1]$, we can write $a=c^{q-1}$ for some $c\in\F_{q^m}^*$, and since $c=\alpha^i$ for some $i\in\Z$, $a=\alpha^{i(q-1)}$. For any $a,b\in[1]$, we then have $a=\alpha^{i(q-1)}$ and $b=\alpha^{j(q-1)}$ for some $i,j\in\Z$, so if $\phi(a)=\phi(b)$, we have
\begin{align*}
\phi(a)&=\phi(b)\\
a^{\frac{q^{m-1}-1}{q-1}}&=b^{\frac{q^{m-1}-1}{q-1}}&&\\
\alpha^{i(q^{m-1}-1)}&=\alpha^{j(q^{m-1}-1)}&&\\
i(q^{m-1}-1)&\equiv j(q^{m-1}-1)&&\pmod{q^m-1}\\
i(q^m-1)-i(q+1)&\equiv j(q^m-1)-j(q+1)&&\pmod{q^m-1}\\
i(q+1)&\equiv j(q+1)&&\pmod{q^m-1}\\
\alpha^{i(q+1)}&=\alpha^{j(q+1)}&&\\
a&=b,&&\\
\end{align*}
and so $\phi$ is one-to-one. Since $|[1]|=\llbracket m\rrbracket$, this means $\phi$ is a bijection from $[1]$ to $[1]$.

Next, we consider any set $Z\subseteq[1]$. If $Z$ is right independent, then in the proof for Lemma $\ref{rclose}$ we have $d=n$, and so we find that all of the $b_i$'s are $\F_q$-linearly independent. If $Z$ is right dependent, then similarly, some $a_s$ would be a root of the right minimal polynomial of $\{a_{k_1},\dots,a_{k_{s-1}}\}$, and so $b_s$ would be an $\F_q$-linear combination of $\{b_{k_1},\dots,b_{k_{s-1}}\}$, and the $b_i$'s would be $\F_q$-linearly dependent.

Thus, $Z$ is right independent if and only if $\{b_1,\dots,b_n\}$ is $\F_q$-linearly independent.

For left independence, we follow the same steps in the proof of Lemma $\ref{lclose}$ to find that $Z$ is left independent if and only if $\{b_1,\dots,b_n\}$ is $\F_q$-linearly independent.

Putting both of these together, if we have a set $Z=\{a_i\}_{1\le i\le n}\subset[1]$, where $a_i=\alpha^{k_i(q-1)}$ and let $b_i=\alpha^{k_i}$, then the following are equivalent:
\begin{itemize}
\item $\{a_1,\dots,a_n\}$ is right independent
\item $\{b_1,\dots,b_n\}$ is $\F_q$-linearly independent
\item $\{b_1^{q^{m-1}-1},\dots,b_n^{q^{m-1}-1}\}$ is left independent.
\end{itemize}
Since $\phi(a_i)=b_i^{q^{m-1}-1}$ for all $1\le i\le n$, $Z$ is right independent if and only if $\phi(Z)$ is left independent, and so $\phi$ is a matroid isomorphism between $[1]_r$ and $[1]_l$.
\end{proof}

Now we are ready to put all of these separate parts together to get the following key theorem.

\begin{theorem}
$\Mr$ is isomorphic to $\Ml$.
\end{theorem}

\begin{proof}
We define a map $\Phi:\F_{q^m}\to\F_{q^m}$ by $\Phi[0]=0$, and $\Phi|{[\alpha^i]}=\gamma_i\phi\gamma_i^{-1}$ for $0\le i\le q-2$:

\begin{center}
	$\xymatrix{[1]_r\ar[d]_{\phi}\ar[r]^{\gamma_i}&[\alpha^i]_r\ar@{-->}[d]^{\gamma_i\phi\gamma_i^{-1}}\\
	[1]_l\ar[r]^{\gamma_i}&[\alpha^i]_l}$
\end{center}
Certainly, since $[0]$ is a class containing only $0$, and the set ${0}$ is the only nonempty subset thereof, it is both right and left independent, and so $\Phi$ restricted to $[0]$ is a matroid isomorphism from $[0]$ to $[0]$. We also have that $\gamma_i\phi\gamma_i^{-1}$ is a matroid isomorphism from $[\alpha^i]_r$ to $[\alpha^i]_l$ by construction. We essentially have the following diagram:
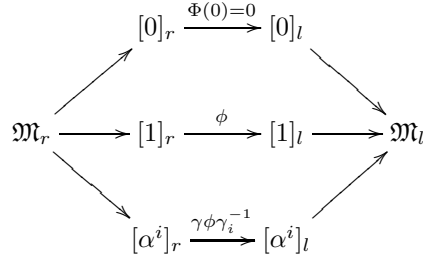
\begin{center}
	$\xymatrix{&[0]_r\ar[r]^{\Phi(0)=0}&[0]_l\ar[dr]&\\
	\Mr\ar[ur]\ar[r]\ar[dr]&[1]_r\ar[r]^{\phi}&[1]_l\ar[r]&\Ml\\
	&[\alpha^i]_r\ar[r]^{\gamma\phi\gamma_i^{-1}}&[\alpha^i]_l\ar[ur]&}$
\end{center}

We first claim that this is a bijection. Certainly, if $\Phi(a)=\Phi(b)$, then \linebreak$[\Phi(a)]=[\Phi(b)]$, and since $[\Phi(a)]=[a]$ for all $a\in\F_{q^m}$, we find that $[a]=[b]$, and since $\Phi$ restricted to each class is a bijection, $a=b$. This proves that $\Phi$ is one-to-one, and since $\F_{q^m}$ is finite, we have that $\Phi$ is a bijection.

Finally, consider any subset $Z$ of $\F_{q^m}$. If $Z$ is right independent, then certainly $Z\cap[\alpha^i]$ is right independent for $0\le i\le q-2$, as is $Z\cap[0]$. Then by the individual sub-matroid isomorphisms, $\Phi(Z\cap[\alpha^i])$ and $\Phi(Z\cap[0])$ are left independent. This means that $\Phi(Z\cap[\alpha^i])\subseteq[\alpha^i]$ is a subset of a left basis for $[\alpha^i]$, and $\Phi(Z\cap[0])\subseteq[0]$ is a subset of a left basis for $[0]$. Then by Corollary 4.4 in $\cite{Wedderburn}$, the union of these bases for $[\alpha^i]$ for each $0\le i\le q-2$ and $[0]$ is a left basis for $\F_{q^m}$, and so the union
\[\Phi(Z\cap[0])\cup\bigcup_{0\le i\le q-2}\Phi(Z\cap[\alpha^i])=\Phi(Z)\]
is a subset of this basis, and thus left independent.

We make the same argument in reverse to see that if $\Phi(Z)$ is left independent, then $Z$ is right independent, and so $\Phi(Z)$ is right independent if and only if $Z$ is right independent, which means that $\Phi$ is a matroid isomorphism and $\Mr$ is isomorphic to $\Ml$.
\end{proof}

\section{Extension Fields}
\label{extension}

In this section, we examine in detail a concept mentioned in Remark 2.8 of \cite{LeroySplitting}. If we have a skew polynomial ring $R$ with coefficients in a field $F$, then we may wish to examine the elements of $R$ as elements of a skew polynomial ring over a field extension of $F$. The following theorem allows us to do this.

\begin{theorem}
	If we have a skew polynomial ring $R=F[x;\sigma,\delta]$, and another polynomial ring $S=T[x;\gamma,\eta]$ such that $F\subseteq T$ is a field extension, $\gamma|_F=\sigma$, and $\eta|_F=\delta$, then $R$ is a subring of $S$.
\end{theorem}
\begin{proof}
	We have that $R\subseteq S$ since $F\subseteq T$. $R$ is additively closed in $S$. For multiplication, we have to worry not only about $F$ being multiplicatively closed, but about the action of the new automorphism $\gamma$, and its $\gamma$-derivation, $\eta$. However, since we know $\gamma|_F=\sigma$ and $\eta|_F=\delta$, and $\sigma$ and $\delta$ both map into $F$, we know that whenever $\gamma$ and $\eta$ are applied to an element of $F$ in the multiplication of two polynomials from $R$, the resulting coefficient will be in $F$, and so the product polynomial will also be in $R$.
	
	Finally, $R$ must contain the multiplicative identity of $S$. In skew polynomial rings, the multiplicative identity is always just $1_K$. In the case of field extensions, $1_F=1_T$ so, $1_T\in R$, and so $R$ is indeed a subring of $S$.
\end{proof}

In the case that $F$ is a finite field, we construct such a ring $S$.

\begin{theorem}
	Let $R=\F_{q^m}[x;\sigma;\delta]$ be a skew polynomial ring with $\sigma\ne1$ and $\T\supseteq\F_{q^m}$ be a finite extension field. Then there exist an automorphism $\gamma$ and a $\gamma$-derivation $\eta$ such that $S=\T[x;\gamma,\eta]$ is a skew polynomial ring and $R$ is a subring of $S$.
\end{theorem}

\begin{proof}
	Given any field extension $\T\supset\F_{q^m}$ and an automorphism $\sigma$ of $\F_{q^m}$, it is a basic result of field theory that we can extend this to an automorphism $\gamma$ of $\T$ such that $\gamma|_{\F_{q^m}}\equiv\sigma$. However, we must also consider $\delta$.
	
	We show that $\delta$ is an inner derivation of the form $\delta(a)=d(a-\sigma(a))$, as seen in the introduction to \cite{BoucherDerivation} (via Chapter 8, Theorem 3.1. of \cite{CohnRings}). Note that since $\F_{q^m}$ is a field, $ab=ba$ for any $a,b\in\F_{q^m}$. Thus, we may write
	
	\begin{align*}
	\delta(ab)&=\delta(ba)\\
	\sigma(a)\delta(b)+\delta(a)b&=\sigma(b)\delta(a)+\delta(b)a\\
	\delta(a)(b-\sigma(b))&=\delta(b)(a-\sigma(a))
	\end{align*}
	
	Consider some $b\in\F_{q^m}$ such that $\sigma(b)\ne b$ (which must exist if $\sigma$ is not the identity homomorphism). Then we have $b-\sigma(b)\ne0$, and so we may define $d=\delta(b)(b-\sigma(b))^{-1}$. For every $a\in\F_{q^m}$, the above equality then gives us
	\[\delta(a)=d(a-\sigma(a))\]
	
	In this case, since $d\in\T$, we let $\eta(a)=d(a-\gamma(a))$. Then for any $a\in\F_{q^m}$, we have $\eta(a)=d(a-\gamma(a))=d(a-\sigma(a))\equiv\delta(a)$, and so $\eta|_{\F_{q^m}}=\delta$. We also must check that this is a $\gamma$-derivation. First, \[\eta(a+b)=d(a+b-\gamma(a+b))=d(a-\gamma(a))+d(b-\gamma(b))=\eta(a)+\eta(b),\]
	so it is an additive homomorphism. Then we note that, for any $a,b\in T$,
	\begin{align*}
	\eta(ab)&=d(ab-\gamma(ab))\\
	&=d(ab-\gamma(a)b)+d(\gamma(a)b-\gamma(ab))\\
	&=\gamma(a)\eta(b)+\eta(a)b.
	\end{align*}
	And so $\eta$ has the property required of a $\gamma$-derivation. Thus, any such $\sigma$-derivation $\delta$ can be extended to a $\gamma$-derivation $\eta$, and we have that there exists a ring $S=\T[x;\gamma,\eta]$ such that $R$ is a subring of $S$.
\end{proof}

We now investigate the right and left roots of polynomials when we  extend the underlying field of a skew polynomial ring with zero derivation. Start with a skew polynomial ring $R=\F_{q^m}[x;\sigma]$, where $\sigma$ is the Frobenius automorphism. Then let $t=km$ for some $k\ge2$, and consider operations in the field $\F_{q^t}$. We note that $\sigma$ can be extended to $\gamma:\F_{q^t}\to\F_{q^t}$. That is, $\gamma|_{\F_{q^m}}=\sigma$. In fact, we can write $\gamma(a)=a^q$, and this is still an automorphism for $\F_{q^t}$, and when restricted to $\F_{q^m}$ is just $\sigma$, so this is what we shall use. Since $\gamma|_{\F_{q^m}}=\sigma$, any elements of $\F_{q^m}$ that were right roots of a polynomial as an element of $R$ will still be right roots of the same polynomial in $S$.

Now $\gamma^{-1}(a)=a^{q^{t-1}}$, but we expect that left roots should not be affected. This gives rise to the following proposition.

\begin{proposition}
	If $f(x)\in R=\F_{q^m}[x;\sigma]$, where $R$ is a subring of $S$, and $a\in\F_{q^m}$, then $f(a)_l$ when $f(x)$ is considered as an element of $S$ is the same as $f(a)_l$ when $f(x)$ is considered an element of $R$.
\end{proposition}

\begin{proof}
	It is enough to show that $\gamma^{-1}|_{\F_{q^m}}=\sigma^{-1}$. Indeed, for $a\in\F_{q^m}$,
	\begin{align*}
	\gamma^{-1}(a)=a^{q^{t-1}}&=\left(a^{q^{m-1}}\right)^{q^{t-m}}
	=\left(\sigma^{-1}(a)\right)^{q^{km-m}}
	=\left(\sigma^{-1}(a)\right)^{q^{m(k-1)}}
	=\sigma^{-1}(a),
	\end{align*}
	since $(k-1)m$ is a multiple of $m$, and $a^{q^m}=a$ for all $a\in\F_{q^m}$.
\end{proof}

In particular, this means that a polynomial does not gain or lose left roots in the original field when it is considered as an element of a polynomial ring with a field extension.

If we have a set of elements $Z=\{a_1,a_2,\dots,a_n\}$ that is right independent in $R=\F_{q^m}[x;\sigma]$, then for any $i\in\{1,\dots,n\}$, we have that $\mu_{Z\setminus\{a_i\},r}(a_i)_r\ne0$ by Definition $\ref{closure}$. We note that the procedure to calculate the minimal polynomial in the extension $S=\F_{q^t}[x;\gamma]$ depends on $\gamma$, which when restricted to elements in $\F_{q^m}$ such as those in $Z$ is the same as $\sigma$. Therefore, in the extended ring $S$, the minimal polynomial $\mu_{Z\setminus\{a_i\},r}$ will be the same as in $R$, and since right evaluation is the same, we will find that $\mu_{Z\setminus\{a_i\},r}(a_i)_r\ne0$ in $S$ as well, and since this is true for all $i\in\{1,\dots,n\}$, we have that $Z$ is right independent in $S$.

We apply the same reasoning for left independence. Given a set \linebreak$Z=\{a_1,a_2,\dots,a_n\}$ that is left independent in $R$, the construction of the left minimal polynomial for $Z\setminus\{a_i\}$ in $S$ for any $i\in\{1,\dots,n\}$ depends upon $\gamma^{-1}$. But we have just proven that $\gamma^{-1}|_{\F_{q^m}}=\sigma^{-1}$, and so we will arrive at the same minimal polynomial as in $R$, and the left evaluation is the same, so $\mu_{Z\setminus\{a_i\},l}(a_i)_l\ne0$ for all $i\in\{1,\dots,n\}$, and the set $Z$ is also left independent in $S$.

\subsection{Splitting Fields}
\label{splitting}
Let $f(x)$ be any polynomial in a skew polynomial ring $R$. This polynomial may be written as $f(x)=\sum_{i=0}^nf_ix^i$ for some $n\in\N$. We then recall that the corresponding right evaluation polynomial is
\[\ov{f^r}(y)=\sum_{i=0}^nf_iy^{\llbracket i\rrbracket}.\]
Upon taking the formal derivative of this polynomial, we have
\[\ov{f^r}'(y)=\sum_{i=1}^n\llbracket i\rrbracket f_iy^{\llbracket i\rrbracket-1}=\sum_{i=1}^nf_iy^{\llbracket i\rrbracket-1}\]
because for $j>0$, $q\mid q^j$, and so $\llbracket i\rrbracket\equiv 1\pmod{q}$.

Multiplying by $y$ gives us
\[y\ov{f^r}'(x)=y\sum_{i=1}^nf_iy^{\llbracket i\rrbracket-1}=\sum_{i=1}^nf_iy^{\llbracket i\rrbracket}=-f_0+\sum_{i=0}^nf_iy^{\llbracket i\rrbracket}=-f_0+\ov{f^r}(y),\]
and so we have that $\ov{f^r}(y)=y\ov{f^r}'(y)+f_0$.

If we let $K_f$ be the splitting field of the polynomial $\ov{f^r}(y)$ over $\F_{q^m}$, then in $K_f[y]$, the polynomial $\ov{f^r}(y)$ factors into linear terms. We examine the structure of the resulting roots.

\begin{theorem}\label{unroots}
	If $n=\deg f(x)$ and $k=\min_{i\in\N}\{i|f_i\ne0\}$, then in $K_f[x;\sigma]$, $f(x)$ has $\llbracket n-k\rrbracket$ distinct nonzero right roots, each with multiplicity $q^k$.
\end{theorem}

\begin{proof}
	We start by considering the case where $k=0$. That is, when $f_0\ne0$. We examine $\ov{f^r}(y)$. If it has a repeated root, then by Theorem 1.68 in $\cite{LidlNiederreiter}$, this root will be a root of both $\ov{f^r}(y)$ and $\ov{f^r}'(y)$. Since $\ov{f^r}(y)=y\ov{f^r}'(y)+f_0$, if $a$ is a repeated root, we would have $\ov{f^r}(a)=0=\ov{f^r}'(y)$, which would mean $f_0=0$, so if $f_0\ne0$, there cannot be any repeated root. This means that $\ov{f^r}(y)$ has $\llbracket n\rrbracket=\frac{q^n-1}{q-1}$ distinct roots in $K_f[y]$. Since $\ov{f}(y)\in K_f[y]$ is the polynomial for the evaluation of $f(x)\in K_f[x;\gamma]$, we know that in $K_f[x;\gamma]$, $f(x)$ has $\llbracket n\rrbracket$ distinct roots, each with multiplicity $1$.
	
	If instead we consider $k>0$, then we have $f_i=0$ for all $i<k$ and $f_k\ne0$. We then write
	\begin{align*}
	\ov{f^r}(y)&=\sum_{i=k}^nf_iy^{\llbracket i\rrbracket}
	=y^{\llbracket k\rrbracket}\sum_{i=k}^nf_iy^{\llbracket i\rrbracket-\llbracket k\rrbracket}\\
	&=y^{\llbracket k\rrbracket}\left(\sum_{i=k}^n\sigma^{-k}(f_i)y^{\llbracket i-k\rrbracket}\right)^{q^k}
	=y^{\llbracket k\rrbracket}\left(\sum_{i=0}^{n-k}\sigma^{-k}(f_{i+k})y^{\llbracket i\rrbracket}\right)^{q^k}.
	\end{align*}
	And so we see that this polynomial has a root of $0$ with multiplicity $\llbracket k\rrbracket$, and then has $q^k$ copies of a polynomial with $\llbracket n-k\rrbracket$ distinct roots, and so every nonzero root has the same multiplicity, $q^k$.
\end{proof}

By noting that $\sigma$ is an automorphism, we see that $f(x)$ is divisible by $x^k$ on the left, and so $f(x)=x^kf_2(x)$, where $f_2(x)$ has the greatest possible number of roots, $\llbracket n-k\rrbracket=\llbracket \deg f_2(x)\rrbracket$. This motivates a definition for the splitting field of a skew polynomial. However, before we define the splitting field, we note that the above work was done with reference to right evaluation. A natural question is what happens to the left roots. We begin with a useful lemma

\begin{lemma}
	Given a right independent set $Z$, if we define the function $\phi_Z$ by $\phi_Z(a)=\sigma(\mu_{Z\setminus\{a\},r}(a)_r)a(\mu_{Z\setminus\{a\},r}(a)_r)^{-1}=a(\mu_{Z\setminus\{a\},r}(a)_r)^{q-1}$, then
	\begin{equation}\label{e:min_polys}\mu_{Z,r}(x)=\mu_{\phi_Z(Z),l}(x).
	\end{equation}
\end{lemma}

\begin{proof}
	We prove this by induction on $|Z|$. Let $b_i=a_i(\mu_{Z\setminus\{a_i\},r}(a_i)_r)^{q-1},$
	so that $\phi_Z(Z)=\{b_1,\dots,b_n\}$. If $|Z|=1$, then
    $Z=\{a_1\}$ for some $a_1\in K$ and $\mu_{Z,r}(x)=(x-a_1)$ is its right minimal polynomial. Note that since $Z\setminus\{a_1\}=\emptyset$, we have $\mu_{Z\setminus\{a_1\},r}(x)=\mu_{\emptyset,r}(x)=1$. Thus, \[b_1=\sigma(\mu_{Z\setminus\{a_1\},r}(a_1)_r)a_1(\mu_{Z\setminus\{a_1\},r}(a_1)_r)=\sigma(1)a_1\cdot1^{-1}=a_1,\] which means $\phi_Z(Z)=\{a_1\}$, and $\mu_{\phi_Z(Z),l}(x)=(x-a_1)$, and so we indeed have $\mu_{Z,r}(x)=\mu_{\phi_Z(Z),l}(x).$
	
	Next, let $n\ge2$ and assume that \eqref{e:min_polys} holds for any right independent set $Z\in\Is$ such that $|Z|=n-1$. Let $Z\in\Is$ be a right independent set with $n$ elements. For $a_1\in Z$, let $Z'=Z\setminus\{a_1\}$ and $\phi_{Z'}(Z')=\{b_2',\dots,b_n'\}$. Since $|Z'|=n-1$, then it holds that 
	\[\mu_{Z',r}(x)=\mu_{\phi_{Z'}(Z'),l}(x).\]
	
	By  interpolation, we have
	\begin{align*}
	\mu_{Z,r}(x)&=(x-a_1(\mu_{Z',r}(a_1)_r)^{q-1})\mu_{Z',r}(x)=(x-b_1)\mu_{Z',r}(x)%
	\end{align*}
	We further note that for any $2\le i\le n$, we may write
	\begin{align*}
	\mu_{Z',r}(x)&=(x-a_i(\mu_{Z'\setminus\{a_i\},r}(a_i)_r)^{q-1})\mu_{Z'\setminus\{a_i\},r}(x)
	=(x-b_i')\mu_{Z'\setminus\{a_i\},r}(x)
	\end{align*}
    Thus, we may write
	\[\mu_{Z,r}(x)=(x-b_1)(x-b_i')\mu_{Z'\setminus\{a_i\},r}(x)\]
	However, we may also write
	\begin{align*}
	\mu_{Z,r}(x)&=(x-b_i)\mu_{Z\setminus\{a_i\},r}(x)\\
	&=(x-b_i)(x-a_1(\mu_{Z\setminus\{a_i,a_1\},r}(a_1)_r)^{q-1})\mu_{Z\setminus\{a_i,a_1\},r}(x)
	\end{align*}
	
	Noting that $Z\setminus\{a_i,a_1\}=Z'\setminus\{a_i\}$ and using $b_1'$ to denote the constant in the second term, we have that
	\[\mu_{Z,r}(x)=(x-b_i)(x-b_1')\mu_{Z'\setminus\{a_i\},r}(x).\]
	
	Since in the division of $\mu_{Z,r}(x)$ on the right by $\mu_{Z'\setminus\{a_i\},r}(x)$, there will be a unique quotient (which we will call $p_i(x)$), we have that
	\[(x-b_1)(x-b_i')=p_i(x)=(x-b_i)(x-b_1'),\]
	so $(x-b_1)|_lp_i(x)$ and $(x-b_i)|_lp_i(x)$, and so $b_1$ and $b_i$ are both left roots of $p_i(x)$. If we had $b_1=b_i$, then since
	\[(x-b_1)\mu_{Z\setminus\{a_1\},r}(x)=\mu_{Z,r}(x)=(x-b_i)\mu_{Z\setminus\{a_i\},r}(x)\]
	we would have by the uniqueness of the quotient of division of $\mu_{Z,r}(x)$ on the left by $(x-b_1)=(x-b_i)$ that
	\[\mu_{Z\setminus\{a_1\},r}(x)=\mu_{Z\setminus\{a_i\},r}(x)\]
	But this would in turn imply that
	\[\mu_{Z\setminus\{a_1\},r}(a_1)_r=\mu_{Z\setminus\{a_i\},r}(a_1)_r=0,\]
	which contradicts $Z$ being right independent. Thus, $b_1\ne b_i$, and so $\{b_1,b_i\}$ is left independent. The left minimal polynomial of this set is the monic polynomial of least degree with both $b_1$ and $b_i$ as left roots, and so we find by our previous work that $\mu_{\{b_1,b_i\},l}(x)=p_i(x)$.
	
	Since $\mu_{\phi_Z(Z),l}(x)$ has $b_1$ and $b_i$ as left roots, it must thus be divisible on the left by their minimal polynomial, and so $p_i(x)|_l\mu_{\phi_Z(Z),l}(x)$ for $2\le i\le n$. Thus, we may write
	\begin{align*}
	\mu_{\phi(Z),l}(x)&=p_i(x)g_i(x)\\
	&=(x-b_1)(x-b_i')g_i(x).
	\end{align*}
	Again, using the uniqueness of the quotient in left division, we find that we must have $\mu_{\phi_Z(Z),l}(x)=(x-b_1)g(x)$, where $g(x)=(x-b_i')g_i(x)$ for all $2\le i\le n$. But then we have that $g(x)$ has all $b_i'$ as left roots, so $\mu_{\phi_{Z'}(Z'),l}(x)|_lg(x)$. Since $\mu_{\phi_{Z'}(Z'),l}(x)=\mu_{Z',r}(x)$ by assumption, and $\deg(\mu_{Z',r}(x))=n-1$, we have that $\deg(g(x))\ge n-1$. However, since $\mu_{\phi_Z(Z),l}(x)=(x-b_1)g(x)$, we also have $\deg(g(x))=\deg(\mu_{\phi(Z),l}(x))-1\le n-1$, and so $\deg(g(x))=n-1$. Comparing the leading coefficients, we see that $g(x)$ must be monic, and so this means that $g(x)=\mu_{\phi_{Z'}(Z'),l}(x)$.
	
	Thus, we have that
	\begin{align*}
	\mu_{\phi_Z(Z),l}(x)&=(x-b_1)g(x)\\
	&=(x-b_1)\mu_{\phi_{Z'}(Z'),l}(x)\\
	&=(x-b_1)\mu_{Z',r}(x)\\
	&=\mu_{Z,r}(x).
	\end{align*}
	And since we have now shown that $\mu_{\phi(Z),l}(x)=\mu_{Z,r}(x)$ for any set $Z$ that is right independent with $|Z|=n$, by induction it is true that $\mu_{\phi(Z),l}(x)=\mu_{Z,r}(x)$ for any right independent set $Z$.
\end{proof}

\begin{corollary}
In $K_f[x;\sigma]$ as defined above, $f(x)$ also has $\llbracket n-k\rrbracket$ distinct nonzero left roots.
\end{corollary}
\begin{proof}
As above, we note that $f(x)=x^kf_2(x)$, where $f_2(x)$ has the greatest possible number of right roots, $\llbracket n-k\rrbracket$. This means that we may select a subset of $Z$ independent right roots of $f_2(x)$, such that $|Z|=n-k$. Then by divisibility properties, we have that $\mu_{Z,r}(x)|_rf_2(x)$, and so we may write $f(x)=x^kc\mu_{Z,r}(x)$ for some $c\ne0$ (since $\deg f_2(x)=\deg\mu_{Z,r}(x)$). From this we may write that
\begin{align*}
f(x)&=x^kc\mu_{Z,r}(x)\\
&=x^kc\mu_{\phi_Z(Z),l}(x),
\end{align*}
Where $\mu_{\phi_Z(Z),l}(x)$ has $\llbracket n-k\rrbracket$ distinct nonzero left roots. For any left root $b$ of $\mu_{\phi_Z(Z),l}(x)$, we may write
\begin{align*}
f(x)&=x^kc\mu_{\phi_Z(Z),l}(x)\\
&=x^kc(x-b)g_b(x)\\
&=x^k(x-bc\sigma^{-1}(c)^{-1})\sigma^{-1}(c)g_b(x)\\
&=(x-\sigma^k(bc\sigma^{-1}(c)^{-1}))x^k\sigma^{-1}(c)g_b(x),
\end{align*}
which means that $\sigma^k(bc\sigma^{-1}(c)^{-1})$ is a left root of $f(x)$. If we have $b_1\ne b_2$, then $\sigma^k(b_1c\sigma^{-1}(c)^{-1})\ne\sigma^k(bc\sigma^{-1}(c)^{-1})$ because $\sigma$ is an automorphism. This means that $f(x)$ has $\llbracket n-k\rrbracket$ distinct nonzero left roots.
\end{proof}

\begin{remark}
It may seem that since left evaluation polynomials have a higher degree than right evaluation polynomials, examining the splitting field of the left evaluation polynomial would result in a ring where the polynomial has more thank $\llbracket n-k\rrbracket$ nonzero left roots. However, by factoring $f(x)$ as $f_2'(x)x^k$, and applying the reasoning of Lemma $\ref{lclose}$ to $f_2'(x)$, we see that there is a limit of $\llbracket n-k\rrbracket$ left roots, just as there is for right roots.

There is thus a discrepancy between the degree of the evaluation polynomial for $f_2'(x)$ (which by the same reasoning as Theorem $\ref{unroots}$, has no repeated roots), and the actual number of left roots when $f_2'(x)$ is considered as an element of $K_f'[x;\gamma]$, where $K_f'$ is the splitting field of $\ov{f_2'^l}(y)$. This can be explained by noting that in $K_f'[x;\gamma]$, while $\gamma^{-1}|_K=\sigma^{-1}$, the two functions do not have the same order, and so the left evaluation polynomial for $f_2(x)$ is no longer $\ov{f_2'^l}(y)$.
\end{remark}

\begin{example}
In $\F_{2^3}[x;\sigma]$, where $\sigma(a)=a^2$ for all $a\in\F_{2^3}$, we find that the polynomial $f=x^2+1$ has $1$ as its only right root and only left root. Since  $\ov{f^r}(y)=y^3+1$, we find that the splitting field $K_f$ is $\F_{2^6}$, where $f(x)$ still has right evaluation polynomial $y^3+1$, which splits as $(y-1)(y-b^{21})(y-b^{42})$, where $b$ is the primitive element of $K_f=\F_{2^6}$ and $b^9=a$ is the primitive element of $\F_{2^3}$. This means that $1,b^{21},$ and $b^{42}$ are right roots of $x^2+1$. As it happens, these are also the left roots, and we see that the polynomial has the greatest possible number of left and right roots given its degree.

The left evaluation polynomial of $f(x)$ over $\F_{2^3}$ is $\ov{f^l}(y)=y^5+1$. The splitting field of this polynomial is $\F_{2^{12}}$. Since we now have $m=12$, in this new field, the left evaluation polynomial for $x^2+1$ is $y^{2049}+1$. This means that in the new field, the left evaluation polynomial does not split completely. In fact, its only roots are $1,c^{1365},$ and $c^{2730}$, where $c$ is the primitive element of $\F_{2^{12}}$. Since we have $c^{65}=b$ from before, we note that this polynomial still has right roots $1,b^{21},$ and $b^{42}$. These are also the left roots, and so while the field is larger, we have not gained any roots for the polynomial.
\end{example}

It is precisely because the splitting field of the right evaluation polynomial gives us the most right and left roots possible that it is properly referred to as the splitting field for $f(x)$.

\begin{definition}
	For a polynomial $f(x)\in F_{q^m}[x;\sigma]$ with $\deg f(x)=n$ and \linebreak$k=\min_{i\in\N}\{i|f_i\ne0\}$, the splitting field $K_f$ is the smallest field such that in $K[x;\sigma]$, $f(x)$ has $\llbracket n-k\rrbracket$ distinct nonzero right roots.
\end{definition}

Next, we note the relationship between $K_f$, the splitting field of the evaluation polynomial $\ov{f^r}(y)$, and $K_f'$, the splitting field of the linearized polynomial $\ov{f}(y)=\sum_{i=1}^nf_iy^{q^i}$. An algorithm for the factorization of the latter can be found in \cite{Splitting}.

\begin{lemma}
Let $K_f$ be the splitting field of $\ov{f^r}(y)$, and $K_f'$ be the splitting field of $\ov{f}(y)=\sum_{i=0}^nf_iy^{q^i}$. Then $K_f\subseteq K_f'$, and $K_f\subsetneq K_f'$ if and only if there is some $a\in K_f$ such that $\ov{f^r}(a)_r=0$, but $a\notin[1]$.
\end{lemma}

\begin{proof}
Note that we may write
\[y\ov{f^r}(y^{q-1})
=y\sum_{i=0}^nf_i(y^{q-1})^{\frac{q^i-1}{q-1}}
=y\sum_{i=0}^nf_iy^{q^i-1}
=\sum_{i=0}^nf_iy^{q^i}=\ov{f}(y).\]

By definition, in $K_f$ we may write $\ov{f^r}(y)=\prod_{i=1}^n(y-a_i)$, where $a_i\in K_f$ for $1\le i\le n$. This means we may write
\[\ov{f}(y)=y\ov{f^r}(y^{q-1})=y\prod_{i=1}^n(y^{q-1}-a_i).\]
This polynomial in turn splits over $K_f'$. In particular, for each $1\le i\le n$, we may write $(y^{q-1}-a_i)=\prod_{\lambda\in\F_q}(y-\lambda b_i)$, where $b_i\in K_f'$ is such that $b_i^{q-1}=a_i$ for all $1\le i\le n$. Since $b_i\in K_f'$, $b_i^{q-1}=a_i\in K_f'$. This means that $\ov{f^r}(y)$ splits over $K_f'$, and so we have $K_f\subseteq K_f'$.

For the second part, if there exists some $a\in K_f$ such that $\ov{f^r}(a)_r=0$, but $a\notin[1]$, then there is no $b\in K_f$ such that $b^{q-1}=a$. From the above factorization, this means that $b\in K_f'$, but $b\notin K_f$, so $K_f\subsetneq K_f'$. For the other direction, if there is no such $a$, then for each $1\le i\le n$, we know that $a_i\in[1]$, and so there is some $b_i\in K_f$ such that $b_i^{q-1}=a_i$, and thus $\lambda b_i\in K_f$ for all $\lambda\in\F_q$ and $1\le i\le n$, and thus $\ov{f}(y)$ splits over $K_f$, which means $K_f'\subseteq K_f$ and the two fields must be equal.
\end{proof}

We now consider the classes in which roots of a polynomial reside. Since $\F_{q^m}$ has $q^m$ elements, besides the $0$ element, there are $q-1$ conjugacy classes of elements with $\frac{q^m-1}{q-1}=\llbracket m\rrbracket$ elements in each. We first note an important property about these values

\begin{lemma}
	For any integers $s,d\in\N$, if $s\le d$, then $(q-1)\llbracket s\rrbracket<q^d$.
\end{lemma}

\begin{proof}
	We note that $\llbracket s\rrbracket=\frac{q^s-1}{q-1}$, and so $(q-1)\llbracket s\rrbracket=q^s-1\le q^d-1<q^d$.
\end{proof}

Next we group the roots of $f(x)$ in $K_f[x;\sigma]$ by conjugacy classes. Noting as before that we may factor out $x^k$, the maximum number of independent nonzero roots of $f(x)$ is $n-k$. We claim that all of these are in the same class.

\begin{theorem}
	Given a polynomial $f(x)\in R$, in the skew polynomial ring using the splitting field $K_f$, $f(x)$ has $\llbracket n-k\rrbracket$ distinct right roots, all in the same conjugacy class in $K_f$.
\end{theorem}

\begin{proof}
	Let $Z$ be the set of right roots of $f(x)\in K_f[x;\sigma]$. Then if $\alpha$ is a primitive element of $K_f$, the conjugacy classes are $[\alpha^i]$ for $i\in\{0,\dots,q-2\}$. We then let $s=\max_{i\in\{0,\dots,q-2\}}\rank(Z\cap [\alpha^i])$. Then there are at most $s$ independent right roots in each of the $q-1$ classes. If there are $s$ independent right roots in a given class, then there are $\llbracket s\rrbracket$ total right roots in that class, and so the maximum possible number of right roots would be $(q-1)\llbracket s\rrbracket$. If we assume that $s\le n-k-1$, then we have $(q-1)\llbracket s\rrbracket<q^{n-k-1}$ from the lemma above. However, we know that $f(x)$ has $\llbracket n-k\rrbracket=\sum_{i=0}^{n-k-1}q^i>q^{n-k-1}$ right roots in $K_f[x;\sigma]$, so it must be that $s=n-k$, and so all of the roots of $f(x)$ are in one class in $K_f$.
\end{proof}

\section{Conclusions}
\label{conclusions}
We have seen that when $\Mr$ and $\Ml$ are the matroids defined by independent sets of roots of polynomials in a skew polynomial ring over $\F_{q^m}$, then there is a matroid bijection between $\Mr$ and $\Ml$. It maps elements from $[1]$ to $[1]$ via $\phi(a)=a^{\frac{q^{m-1}-1}{q-1}}$, and applies to elements of other classes by first mapping to $[1]$ and then mapping back to the original class. Then it was also proved that if $\T$ is a finite extension field of $\F_{q^m}$, we can construct a matching automorphism $\gamma$ and $\gamma$-derivation $\eta$ so that $S=\T[x;\gamma,\eta]$ is a skew polynomial ring with $R$ as a subring. In fact, this can be done in such a way that all of the roots of a given polynomial in $R$ are now in a single class. This opens up several doors for examining the deeper structure of roots.

\pagebreak
\bibliographystyle{siam}
\bibliography{bibliography}

\end{document}